\documentclass[12pt]{amsart}
\usepackage{amsmath, amsfonts, amssymb, amsthm,hyperref,mathtools,array}
\hypersetup{hypertex=true,
	colorlinks=true,
	linkcolor=blue,
	anchorcolor=blue,
	citecolor=blue}
\usepackage{bm}
\allowdisplaybreaks[4]
\def\({\bg(}
\def\){\bg)}

\def\pmod #1{\ ({\rm{mod}}\ #1)}

\def\qbinom #1#2#3{{\genfrac{[}{]}{0pt}{}{#1}{#2}}_{#3}}

\theoremstyle{plain}
\newtheorem{theorem}{Theorem}[section]
\newtheorem{lemma}{Lemma}
\newtheorem{corollary}{Corollary}
\newtheorem{conjecture}{Conjecture}

\theoremstyle{definition}

\theoremstyle{remark}

\makeatletter
\@namedef{subjclassname@2020}{%
	\textup{2020} Mathematics Subject Classification}
\makeatother
\vspace{4mm}

\allowdisplaybreaks

\begin{document}
	\medskip
	
	\title[Supercongruences Arising from Truncated Appell Series]
	{Supercongruences Arising from Truncated Appell Series}
	\author[ H.-X. Ni]{ He-Xia Ni*}

\address{(He-Xia Ni) Department of applied mathematics, Nanjing Audit University, Nanjing 211815, People's Republic of China}
\email{\tt nihexia@yeah.net}

	\keywords{ $q$-supercongruence; truncated Appell series;
		the cyclotomic polynomial; transformation formula.
		\newline \indent This research was supported by the National Natural Science Foundation of China (Grant No. 12371004).
		\newline \indent *Corresponding author}
	
	\begin{abstract}
		In recent years, both Appell series and truncated Appell series have garnered significant research interest among scholars. In this paper, we investigate the congruence properties of truncated Appell series and truncated $q$-Appell series. By employing various combinatorial identities and the `creative microscoping' method introduced by Guo and Zudilin, we provide $q$-analogues for several known congruences of truncated Appell series $F_1$ and $F_2$. Meanwhile, we establish a new $q$-supercongruence related to the truncated Appell series $F_4$, which provides a $q$-analogue for two conjectures by Apagodu and Zeilberger.
		
	\end{abstract}
	\maketitle

	\section{Introduction}
	\setcounter{lemma}{0}
\setcounter{theorem}{0}
\setcounter{equation}{0}
\setcounter{conjecture}{0}
\setcounter{remark}{0}
\setcounter{corollary}{0}

\subsection{Background and motivation}
In 1880, Appell  introduced four types of functions, which are double series in two variables obtained by extending the Gauss series through an increase in the number of variables. Such series are known as Appell series. 
The Appell series $F_1, F_2$, $F_3$ and $F_4$ are defined as follows:
\begin{align}
	F_{1}[a;b,b';c;x,y]&=\sum_{i,j=0}^{\infty}\frac{(a)_{i+j}(b)_{i}(b')_{j}}{i!j!(c)_{i+j}}x^iy^j,\label{Appellseries01}\\
	F_{2}[a;b,b';c,c';x,y]&=\sum_{i,j=0}^{\infty}\frac{(a)_{i+j}(b)_{i}(b')_{j}}{i!j!(c)_{i}(c')_j}x^iy^j,\label{Appellseries02}\\
	F_{3}[a,a';b,b';c;x,y]&=\sum_{i,j=0}^{\infty}\frac{(a)_{i}(a')_{j}(b)_{i}(b')_{j}}{i!j!(c)_{i+j}}x^iy^j,\label{Appellseries03}\\
	F_{4}[a,b;c,c';x,y]&=\sum_{i,j=0}^{\infty}\frac{(a)_{i+j}(b)_{i+j}}{i!j!(c)_{i}(c')_j}x^iy^j,\label{Appellseries04}
\end{align}
where $(a)_k=a(a+1)\cdots (a+k-1)$ for $k\geq 1$ and $(a)_0=1$ is the \textit{shifted factorial}.

On the other hand, many highly interesting and novel expansions of Appell's double hypergeometric functions have been obtained and discussed. These expansions were derived using operative functions constructed from hypergeometric functions, in which the parameters were replaced by the partial differential operators $x \frac{\partial}{\partial x}$ and $y \frac{\partial}{\partial y}$. In 1942, Jackson \cite{Jackson1942} defined the functions  $\Phi^{(1)}, \Phi^{(2)},\Phi^{(3)},\Phi^{(4)}$ via double hypergeometric series as follows:
\begin{align*}
		\Phi^{(1)}[a;b,b';c;q;x,y]_{\infty}&=\sum_{i,j=0}^{\infty}\frac{(a;q)_{i+j}(b;q)_{i}(b';q)_{j}}{(q;q)_i(q;q)_j(c;q)_{i+j}}x^iy^j,\notag\\
		\Phi^{(2)}[a;b,b';c,c';q;x,y]_{\infty}&=\sum_{i,j=0}^{\infty}\frac{(a;q)_{i+j}(b;q)_{i}(b';q)_{j}}{(q;q)_i(q;q)_j(c;q)_{i}(c';q)_{j}}x^iy^j,\notag\\
		\Phi^{(3)}[a,a';b,b';c;q;x,y]_{\infty}&=\sum_{i,j=0}^{\infty}\frac{(a;q)_{i}(a';q)_{j}(b;q)_{i}(b';q)_{j}}{(q;q)_i(q;q)_j(c;q)_{i+j}}x^iy^j,\notag\\
		\Phi^{(4)}[a,b;c,c';q;x,y]_{\infty}&=\sum_{i,j=0}^{\infty}\frac{(a;q)_{i+j}(b;q)_{i+j}}{(q;q)_i(q;q)_j(c;q)_{i}(c';q)_{j}}x^iy^j,
\end{align*}
 which serve as $q$-analogues of the Appell functions. Additionally, he demonstrated that these basic functions satisfy certain $q$-partial differential equations. In the limiting cases, these basic equations reduce to the ordinary linear partial differential equations satisfied by Appell functions \cite[213--214]{Sla1966}. 

In recent years, many scholars have investigated congruence properties related to Appell series. For instance, Lin \cite{Liu2019}, Lin and Liu \cite{LinLiu2020} introduced the truncated  Appell series, defined by setting the upper limit of the double summation in the series \eqref{Appellseries01}--\eqref{Appellseries04} to a finite number $n$, denoted as $F_{1}[a;b,b';c;x,y]_n, F_{2}[a;b,b';c,c';x,y]_n, F_{3}[a,a';b,b';c;x,y]_n, F_{4}[a,b;c,c';x,y]_n$. Similarly, we can define the truncated $q$-Appell series. For example, the truncated  $q$-Appell series $\Phi^{(1)}[a;b,b';c;q;x,y]_{n}, \Phi^{(4)}[a,b;c,c';q;x,y]_{n}$ can be defined as:
\begin{align*}
\Phi^{(1)}[a;b,b';c;q;x,y]_{n}&=\sum_{i,j=0}^{n}\frac{(a;q)_{i+j}(b;q)_{i}(b';q)_{j}}{(q;q)_i(q;q)_j(c;q)_{i+j}}x^iy^j,\notag\\
\Phi^{(4)}[a,b;c,c';q;x,y]_{n}&=\sum_{i,j=0}^{n}\frac{(a;q)_{i+j}(b;q)_{i+j}}{(q;q)_i(q;q)_j(c;q)_{i}(c';q)_{j}}x^iy^j.
\end{align*}

In 2017, Apagodu and Zeilberger \cite[Conjecture 5', Super-Conjecture 5,  Super-Conjecture 5']{AZ2017} investigated algorithms for proving combinatorial identities and proposed several conjectures related to Appell series $F_4$, such as: for any prime $p\geq 5$ and any pair of positive integers $r,s$,
\begin{align}
F_4[1,1;1,1;1,1]_{p-1}&\equiv \bigg(\frac{p}{3}\bigg)\pmod{p^2},\label{ApaZeilberger-conje-1}\\
\sum_{i=0}^{rp-1}\sum_{j=0}^{sp-1}\frac{(1)_{i+j}^2}{i!^2j!^2} &\equiv\bigg(\frac{p}{3}\bigg)\sum_{i=0}^{r-1}\sum_{j=0}^{s-1}\frac{(1)_{i+j}^2}{i!^2j!^2}\pmod{p^2},\label{ApaZeilberger-conje-2}
\end{align}
where $\bigg(\frac{\cdot}{p}\bigg)$ denotes the Legendre symbol modulo $p$. In the period around 2020, Liu \cite{Liu2019}, Lin and Liu \cite{LinLiu2020} investigated congruence properties of truncated Appell series and obtained several elegant results, such as: for any prime $p>3$,
\begin{align}
	F_{1}[\frac{1}{2};\frac{1}{2},\frac{1}{2};1;1,1]_{\frac{p-1}{2}}&\equiv p\pmod{p^2},\label{Appellseries06}\\
	F_{2}[\frac{1}{2};\frac{1}{2},\frac{1}{2};1,1;1,1]_{\frac{p-1}{2}}&\equiv\begin{cases}0 \pmod{p^2}&\text{if }p\equiv 3\pmod{4},\\[5pt]
		-\Gamma_p(\frac{1}{4})^4\pmod{p^2}&\text{if }p\equiv 1\pmod{4}.
	\end{cases}\label{Appellseries07}\\
	F_{4}[\frac{1}{2},\frac{1}{2};1,1;1,1]_{\frac{p-1}{2}} &\equiv\begin{cases}0 \pmod{p^2}&\text{if }p\equiv 2\pmod{3},\\[5pt]
		(-1)^{\frac{p+1}{2}}\Gamma_p(\frac{1}{6})^2\Gamma_p(\frac{1}{3})^2\pmod{p^2}&\text{if }p\equiv 1\pmod{3}.
	\end{cases}\notag
\end{align}

Furthermore, in 2020, Liu and Petrov \cite{LiuPetrov2020} provided a $q$-analogue of \eqref{ApaZeilberger-conje-1} as follows:
\begin{align*}
\sum_{i=0}^{n-1}\sum_{j=0}^{n-1}\frac{(q;q)_{i+j}^2}{(q;q)_i^2(q;q)_j^2}q^{i+j+j^2}	\equiv \bigg(\frac{n}{3}\bigg)q^{\frac{n^2-1}{3}}\pmod{\Phi_{n}(q)^2}.
\end{align*}
Recently, He and Wang \cite{Hewang2026} also provided an alternative $q$-analogue of \eqref{ApaZeilberger-conje-1} through  the $q$-Chu-Vandermonde summation formula and the $q$-Zeilberger algorithm as follows: for any positive integer $n$,
\begin{align}\label{Hewang01}
\Phi^{(4)}[q,q;q,q;q;q,q]_{n-1}	\equiv \bigg(\frac{n}{3}\bigg)q^{\frac{2(n^2-1)}{3}}\pmod{\Phi_{n}(q)^2}.
\end{align}

Here and subsequently, for any complex $a$, the {\it $q$-shifted factorial} \cite{GR} is defined as:   
$$
(a;q)_m=\begin{cases}
	(1-a)(1-aq)\cdots(1-aq^{m-1}) &\text{if }m\geq 1,\\[5pt]
	1 &\text{if }m=0,
\end{cases}
$$ and the {\it $q$-binomial coefficient} is defined by
$$
\qbinom{a}{m}{}=\qbinom{a}{m}q=\begin{cases}
	\dfrac{(q^{a-m+1};q)_{m}}{(q;q)_m} &\text{if }m\geq 1,\\[5pt]
	1 &\text{if }m=0.
\end{cases}
$$
In particular, we set $\qbinom{a}{m}q=0$ if $m<0$  and $[n]=[n]_q=(1-q^n)/(1-q)$. Moreover, the {\it $m$-th cyclotomic polynomial} is defined as
$$
\Phi_m(q)=\prod_{\substack{1\leq k\leq m\\ \gcd(m,k)=1}}(q-e^{2\pi i\cdot\frac{k}{m}}).
$$
To simplify, for $m\in \mathbb{N^{+}}$ and $n\in \mathbb{N}$, we also frequently use the following shortened symbols:
$$(a_1,\ldots, a_m;q)_n=(a_1;q)_n\cdots (a_m;q)_n.$$

Moreover, in 2025, Wang and Yu \cite{WY2025-02} had proposed a set of new congruences and $q$-congruences concerning truncated Appell numbers, including: for any prime $p$ and any positive integer $d$, if $p\equiv 1\pmod{2d}$, then
\begin{align*}
F_{1}[\frac{1}{d};\frac{1}{2d},\frac{1}{2d};1;1,1]_{\frac{p-1}{d}}\equiv\binom{-\frac{2}{d}}{\frac{p-1}{d}}\pmod{p^2}.
\end{align*}
For more related results and the latest progress on truncated Appell numbers, please refer to \cite{GuoZu-a, GuoZu-b, GuoZhu2020,G2024,Hewang2026,WY2025,W2024}.

\subsection{Main results}
Inspired by the aforementioned work, our primary objective is to provide $q$-analogues of formulas \eqref{Appellseries06}--\eqref{Appellseries07}. Moreover, we will present more supercongruences related to the truncated Appell series $F_4$.

\begin{theorem}\label{thm-a}
	Let $d$ and $r$ be integers such that $d\geq 2$ and $\gcd{(r,d)}=1$. Let $n$ be a positive integer satisfying $n\equiv r \pmod{d}$ and $1\leq \frac{n-r}{d} \leq \frac{n-1}{2}$. 
	Then the following results hold.
	
	{\rm (i)} If $d=2, r=1$, then 
	\begin{align}\label{thm-q-01}
		\Phi^{(1)}[q;q,q;q^2;q^2;q^2,1]_{\frac{n-1}{2}}\equiv[n]q^{\frac{n+1}{2}} \pmod{[n]\Phi_n(q)}.
	\end{align}

	{\rm (ii)} If $d>2$, we have  
	\begin{align}\label{thm-q-02}
	&\Phi^{(1)}[q^r;q^r,q^r;q^d;q^d;q^d,q^{d-2r}]_{\frac{n-r}{d}}\notag\\
	&\quad\equiv \frac{q^{r(n-r)/d}(q^{d-2r};q^d)_{\frac{n-r}{d}}}{(q^{d};q^d)_{\frac{n-r}{d}}}\bigg\{1-(1-q^n)\frac{n-r}{d}-[n]\sum_{j=1}^{\frac{n-r}{d}}\bigg(\dfrac{q^{dj-r}}{[dj-r]}+\dfrac{q^{dj-2r}}{[dj-2r]}\bigg)\bigg\} \pmod{\Phi_n(q)^2}.
	\end{align}

	{\rm (iii)} Moreover, modulo $\Phi_n(q)^2$,
	\begin{align}\label{thm-q-03}
	&\sum_{i=0}^{n-1}\sum_{j=0}^{n-1-i}\frac{(q^{d-r};q^{d})_{i+j}(q^{r};q^{d})_{i}(q^{r};q^{d})_{j}}{(q^d;q^d)_{i+j}(q^d;q^d)_i(q^d;q^d)_j}q^{di}\notag\\
	&\quad\equiv-q^{-r(n-r)/d}\frac{(q^{r+d};q^d)_{\frac{n-r-d}{d}}}{(q^{-r+d};q^d)_{\frac{n-r-d}{d}}}\bigg\{(d-1)q^{2r}\frac{1-q^n}{1-q^{2r}}+\frac{1}{n}(q^{d-2r};q^d)_{n-1}\bigg\}.
\end{align}

\end{theorem}
Indeed, \eqref{thm-q-01} and the case of \eqref{thm-q-03} with parameters $(d,r)=(2,1)$ provide two distinct forms of $q$-analogues for the congruence \eqref{Appellseries06}.
Choosing $n=p^s$ be an odd prime power and taking $q\to 1$ in Theorem \ref{thm-a}, we get the following result.

\begin{corollary}
	Let $d$ and $r$ be positive integers such that $d\geq 2, 1\leq \frac{p^s-r}{d} \leq \frac{p^s-1}{2}$  and $\gcd{(r,d)}=1$. Let $p>3$ be an odd prime and $s$ a positive integer with $p^s\equiv r \pmod{d}$. Then,
	\begin{align*}
		F_{1}[\frac{r}{d};\frac{r}{d},\frac{r}{d};1,1;1,1]_{\frac{p^s-r}{d}}&\equiv\begin{cases}\dfrac{(\frac{d-2r}{d})_{(p^s-r)/d}}{(1)_{(p^s-r)/d}}\bigg\{1-p^s\sum_{j=1}^{\frac{p^s-r}{d}}(\frac{1}{dj-r}+\frac{1}{dj-2r})\bigg\}\pmod{p^2} &\text{if }d>2,\\[5pt]
			p^s \pmod{p^{s+1}}&\text{if }d=2.
		\end{cases}\\
\sum_{i=0}^{p^s-1}\sum_{j=0}^{p^s-1-i}\frac{(\frac{d-r}{d})_{i+j}(\frac{r}{d})_{i}(\frac{r}{d})_{j}}{(1)_{i+j}(1)_i(1)_j}	&\equiv-\frac{(d-1)p^s}{2r}\frac{(1+\frac{r}{d})_{\frac{p^s-r-d}{d}}}{(1-\frac{r}{d})_{\frac{p^s-r-d}{d}}}\pmod{p^2}.
	\end{align*}

\end{corollary}

\begin{theorem}\label{thm-aaa}
	Let $n$ be a positive odd integer. Then, modulo $\Phi_{n}(q)$,

	\begin{align}\label{qconaaa1}
		&\sum_{i=0}^{\frac{n-1}{2}}\sum_{j=0}^{\frac{n-1}{2}}\frac{(q;q^{2})_{i+j}(q;q^{2})_{i}(q;q^{2})_{j}(-q;q^{2})_{j}}{(q^2;q^2)_i^2(q^2;q^2)_j^2(-q^2;q^2)_j}q^{2i+2j},\notag\\
		&\quad\equiv\begin{cases}0 &\text{if }n\equiv 3\pmod{4},\\[5pt]
			q^{3(n-1)/4}\dfrac{(q^2;q^{4})_{(n-1)/4}^2}{(q^4;q^{4})_{(n-1)/4}^2}&\text{if }n\equiv 1\pmod{4}.
		\end{cases}
	\end{align}

\end{theorem}

Setting $n=p$ and taking $q\to 1$ in \eqref{qconaaa1}, we catch hold of a partial $q$-analogue of \eqref{Appellseries07}.

\begin{theorem}\label{thm-heliu02}
	Let $n$ be a positive odd integer. Then,  modulo $\Phi_{n}(q)^2$,
	
	\begin{align*}
			&\sum_{i=0}^{\frac{n-1}{2}}\sum_{j=0}^{\frac{n-1}{2}}\frac{(q;q^{2})_{i+j}^2}{(q^2;q^2)_i^2(q^2;q^2)_j^2(-q^2;q^2)_{i+j}^2}q^{2i+2j+2j^2},\notag\\
		&\quad\equiv\begin{cases}0 &\text{if }n\equiv 3\pmod{4},\\[5pt]
			q^{(n-1)/2}\dfrac{(q^2;q^{4})_{(n-1)/4}^2}{(q^4;q^{4})_{(n-1)/4}^2}&\text{if }n\equiv 1\pmod{4}.
		\end{cases}		
	\end{align*}

\end{theorem}
Choosing $n = p$ and then letting $q\to 1$ in the above theorem, we obtain the supercongruence.
\begin{corollary}\label{thm-heliu01}
	Let $p>3$ be an odd prime. Then,  modulo $p^2$,
	
	\begin{align*}
			&F_{4}[\frac{1}{2},\frac{1}{2};1,1;\frac{1}{4},\frac{1}{4}]_{\frac{p-1}{2}}\equiv\begin{cases}0 &\text{if }p\equiv 3\pmod{4},\\[5pt]
				\dfrac{(\frac{1}{2})_{(p-1)/4}^2}{(1)_{(p-1)/4}^2}&\text{if }p\equiv 1\pmod{4}.
		\end{cases}
	\end{align*}
\end{corollary}

Moreover, we present  provides a $q$-analogue  for \eqref{ApaZeilberger-conje-2} due to Apagodu and Zeilberger \cite[ Super-Conjecture 5']{AZ2017}.

\begin{theorem}\label{thm-ApaZeil-qcon-12}
	Let $n, r, s$ be positive integers. Then,  modulo $\Phi_{n}(q)^2$,
	\begin{align}\label{thm-ApaZeil-qcon-2}
		\sum_{i=0}^{rn-1}\sum_{j=0}^{sn-1}\frac{(q;q)_{i+j}^2}{(q;q)_i^2(q;q)_j^2}q^{i+j}\equiv\bigg(\frac{n}{3}\bigg)q^{2(n^2-1)/3}\sum_{i=0}^{r-1}\sum_{j=0}^{s-1}\frac{(q^{n^2};q^{n^2})_{i+j}^2}{(q^{n^2};q^{n^2})_i^2(q^{n^2};q^{n^2})_j^2}q^{n^2(i+j)}.
	\end{align}

\end{theorem}
Moreover, If $r=s$, \eqref{thm-ApaZeil-qcon-2} reduces to the following $q$-supercongruence related to truncated $q$-Appell series $\Phi^{(4)}$:
\begin{align}\label{thm-ApaZeil-qcon-111}
	\Phi^{(4)}[q,q;q,q;q;q,q]_{rn-1}	\equiv \bigg(\frac{n}{3}\bigg)q^{2(n^2-1)/3}\Phi^{(4)}[q,q;q,q;q;q,q]_{r-1}\pmod{\Phi_{n}(q)^2}.
\end{align} 
Setting $n=p\geq5$ be an odd prime and taking $q\to 1$ in \eqref{thm-ApaZeil-qcon-111}, we catch hold of the formula: 
\begin{align*}
	F_{4}[1,1;1,1;1,1]_{rp-1}\equiv\bigg(\frac{p}{3}\bigg)	F_{4}[1,1;1,1;1,1]_{r-1}\pmod{p^2}.
\end{align*}

\subsection{Outline of this paper}
Our subsequent exposition is organized as follows. In the following sections, we will employ the `creative microscoping' method developed by Guo and Zudilin \cite{GuoZu-a}, the Chinese remainder theorem, and some classical  hypergeometric transformation formulas. In particular, in the course of proving Theorem \ref{thm-ApaZeil-qcon-12}, we employ the $q$-differential operator and the properties of $q$-roots of unity to establish several crucial identities.

 In accordance with Gasper and Rahman's  conventions, we define the basic hypergeometric series \cite{GR} by:	
\begin{equation*}
	{}_{r+1}\phi_r\!\left[\begin{matrix}
		a_1,a_2,\dots,a_{r+1}\\b_1,\dots,b_r
	\end{matrix};q,z\right]=\sum_{k=0}^\infty
	\frac{(a_1,a_2,\dots,a_{r+1};q)_k}{(q,b_1,\dots,b_r;q)_k}
	z^k.
\end{equation*}	
To prove Theorem \ref{thm-a}, we will use a reduction formula \cite[(10.3.7)]{GR}:
\begin{align}\label{identiti1}
	&\Phi^{(1)}[q^a;q^b,q^{b'};q^c;q;x,q^{c-a-b'}]=\frac{(q^{c-b'},q^{c-a};q)_{\infty}}{(q^{c},q^{c-a-b'};q)_{\infty}}{}_{2}\phi_{1}\!\left[\begin{matrix}
		q^a,q^b\\q^{c-b'}	\end{matrix};q,x\right].
\end{align}	
Furthermore, we will make repeated use of the $q$-Vandermonde ($q$-Chu-Vandermonde) summation formulae \cite[Appendix (II. 6), (II. 7)]{GR}:
\begin{align}
	{}_{2}\phi_{1}\!\left[\begin{matrix}
		a,q^{-n}\\c	\end{matrix};q,q\right]&=\frac{(c/a;q)_n}{(c;q)_n}a^n,\label{identiti2}\\
	{}_{2}\phi_{1}\!\left[\begin{matrix}
			a,q^{-n}\\c	\end{matrix};q,cq^n/a\right]&= \frac{(c/a; q)_n}{(c; q)_n}, \label{identiti2-1}
\end{align}	
and another form of the $q$-Chu-Vandermonde identity \cite[(3.3.10)]{An1985}:
\begin{align}\label{identiti3}
\qbinom{m+n}{k}{q}=\sum_{j=0}^{k}q^{(m-j)(k-j)}\qbinom{m}{j}{q}\qbinom{n}{k-j}{q}.
\end{align}	
Moreover, in the proof of Theorems \ref{thm-aaa} and \ref{thm-heliu02}, we will employ Andrews' terminating $q$-analogue of Watson's ${}_{3}F_2$ sum \cite[Appendix (II. 17)]{GR}: 
\begin{equation}\label{identiti4}
	{}_{4}\phi_3\!\left[\begin{matrix}
		q^{-n},a^2q^{n+1},c,-c\\aq,-aq,c^2
	\end{matrix};q,q\right]=\begin{cases}0 &\text{if $n$ is odd},\\[5pt]
	\dfrac{c^n(q,a^2q^2/c^2;q^2)_{n/2}}{(a^2q^2,c^2q;q^2)_{n/2}}&\text{if $n$ is even}.
	\end{cases}
\end{equation}

\section{Proof of Theorem \ref{thm-a}}
	\setcounter{lemma}{0}
\setcounter{theorem}{0}
\setcounter{equation}{0}
\setcounter{conjecture}{0}
\setcounter{remark}{0}
\setcounter{corollary}{0}

Before the proof of Theorem \ref{thm-a}, we need to present the following lemmas by employing the method of `creative microscoping' and the $q$-Vandermonde summation formula \eqref{identiti2}.

\begin{lemma}\label{Th1proofLemma11}
	Let $d$ and $r$ be integers such that $d\geq 2$ and $\gcd{(r,d)}=1$. Let $n$ be a positive integer satisfying $n\equiv r \pmod{d}$ and $1\leq \frac{n-r}{d} \leq \frac{n-1}{2}$. Then, modulo $(a-q^n)(b-q^n)$,
	\begin{align}\label{qlemma11cong11111}
		&\Phi^{(1)}[q^r/a;q^r/b,q^r/b;q^d;q^d;q^d,abq^{d-2r}]_{\frac{n-r}{d}}\notag\\
	&\quad\equiv \frac{(b^2q^{d-2r};q^d)_{\frac{n-r}{d}}}{(q^d;q^d)_{\frac{n-r}{d}}}(q^r/b)^{\frac{n-r}{d}}\times \frac{b-q^n}{b-a}\notag\\
	&\quad+\frac{(aq^{d-r};q^d)_{\frac{n-r}{d}}(abq^{d-2r};q^d)_{\frac{n-r}{d}}}{(q^d;q^d)_{\frac{n-r}{d}}(bq^{d-r};q^d)_{\frac{n-r}{d}}}(q^r/a)^{\frac{n-r}{d}}\times \frac{a-q^n}{a-b}.
	\end{align}

\end{lemma}
\begin{proof}
	Firstly, for $a=q^n$, letting $q\mapsto q^d$, and taking $q^a\mapsto q^{r-n}$, $q^b\mapsto q^r/b$, $q^{b'}\mapsto q^r/b$, $q^{c}\mapsto q^d$, $x\mapsto q^d$ in \eqref{identiti1}, we can get that
	\begin{align*}
		&\Phi^{(1)}[q^{r-n};q^r/b,q^r/b;q^d;q^d;q^d,bq^{d-2r+n}]_{\frac{n-r}{d}}\notag\\
		&\quad=\frac{(bq^{d-r},q^{d-r+n};q^{d})_\infty}{(q^d,bq^{d-2r+n};q^{d})_\infty}\sum_{k=0}^{\infty}\frac{(q^{r-n},q^r/b;q^d)_k}{(bq^{d-r},q^d;q^d)_k}q^{dk}\notag\\
		&\quad=\frac{(b^2q^{d-2r};q^d)_{\frac{n-r}{d}}}{(q^d;q^{d})_\frac{n-r}{d}}(q^r/b)^{\frac{n-r}{d}},
	\end{align*}
	where the last identity follows from $q\mapsto q^d$, $n\mapsto \frac{n-r}{d}$, $a\mapsto q^r/b$, $c\mapsto bq^{d-r}$ in \eqref{identiti2}. 
Thus, the $q$-congruence \eqref{qlemma11cong11111} modulo $a-q^n$ is established.

	Secondly, for $b=q^n$, 	performing the replacements $q\mapsto q^d$, and taking $q^a\mapsto q^{r}/a$, $q^b\mapsto q^{r-n}$, $q^{b'}\mapsto q^{r-n}$, $q^{c}\mapsto q^d$, $x\mapsto q^d$ in \eqref{identiti1}, we obtain 
	\begin{align*}
	&\Phi^{(1)}[q^r/a;q^{r-n},q^{r-n};q^d;q^d;q^d,aq^{d-2r+n}]_{\frac{n-r}{d}}\notag\\
	&\quad=\frac{(q^{d-r+n},aq^{d-r};q^{d})_\infty}{(q^d,aq^{d-2r+n};q^{d})_\infty}\sum_{k=0}^{\infty}\frac{(q^{r}/a,q^{r-n};q^d)_k}{(q^{d-r+n},q^d;q^d)_k}q^{dk}\notag\\
	&\quad=\frac{(aq^{d-r};q^d)_{\frac{n-r}{d}}(aq^{d-2r+n};q^d)_{\frac{n-r}{d}}}{(q^d;q^{d})_\frac{n-r}{d}(q^{d-r+n};q^d)_{\frac{n-r}{d}}}(q^r/a)^{\frac{n-r}{d}}.
\end{align*}	
	The last identity follows from $q\mapsto q^d$, $n\mapsto \frac{n-r}{d}$, $a\mapsto q^r/a$, $c\mapsto q^{d-r+n}$ in \eqref{identiti2}. In other words, the $q$-congruence \eqref{qlemma11cong11111} is valid modulo	$b-q^n$.

Finally, since $a-q^n$ and $b-q^n$ are coprime polynomials, by applying the Chinese remainder theorem and the following formula:
	\begin{align*}
		\frac{b-q^n}{b-a}\equiv 1\pmod{a-q^n},\quad\quad\quad\frac{a-q^n}{a-b}\equiv 1\pmod{b-q^n}, 
\end{align*}
we obtain the $q$-congruence \eqref{qlemma11cong11111}.		
\end{proof}
\begin{lemma}\label{Th1proofLemma22}
With the same assumptions as Lemma \ref{Th1proofLemma11}, we have, modulo $(a-q^{(d-1)n})(b-q^n)$,
\begin{align}\label{qcona1}
	&\sum_{i=0}^{n-1}\sum_{j=0}^{n-1-i}\frac{(q^{d-r}/a;q^{d})_{i+j}(q^{r}/b;q^{d})_{i}(q^{r}/b;q^{d})_{j}}{(q^d;q^d)_{i+j}(q^d;q^d)_i(q^d;q^d)_j}q^{di}(ab)^j\notag\\
	&\quad\equiv (q^r/b)^{n-1-\frac{n-r}{d}}\frac{(b^2q^{d-2r};q^d)_{n-1-\frac{n-r}{d}}}{(q^{d};q^d)_{n-1-\frac{n-r}{d}}}\frac{b^{d-1}-q^{(d-1)n}}{b^{d-1}-a}\notag\\
		&\quad+(q^{d-r}/a)^{\frac{n-r}{d}}\frac{(aq^{r};q^d)_{\frac{n-r}{d}}(ab;q^d)_{\frac{n-r}{d}}}{(q^{d};q^d)_{\frac{n-r}{d}}(bq^{d-r};q^d)_{\frac{n-r}{d}}}\frac{a-q^{(d-1)n}}{a-b^{d-1}}.
\end{align}

\end{lemma}
\begin{proof}
Firstly, for $a=q^{(d-1)n}$, letting $q\mapsto q^d$, and taking $q^a\mapsto q^{d-r-(d-1)n}$, $q^b\mapsto q^r/b$, $q^{b'}\mapsto q^r/b$, $q^{c}\mapsto q^d$, $x\mapsto q^d$ in \eqref{identiti1}, we can get that
\begin{align*}
	&\sum_{i=0}^{n-1-\frac{n-r}{d}}\sum_{j=0}^{n-1-\frac{n-r}{d}-i}\frac{(q^{d-r-(d-1)n};q^{d})_{i+j}(q^{r}/b;q^{d})_{i}(q^{r}/b;q^{d})_{j}}{(q^d;q^d)_{i+j}(q^d;q^d)_i(q^d;q^d)_j}q^{di}(bq^{(d-1)n})^j\notag\\
	&\quad=\frac{(bq^{d-r},q^{r+(d-1)n};q^{d})_\infty}{(q^d,bq^{(d-1)n};q^{d})_\infty}\sum_{k=0}^{\infty}\frac{(q^{d-r-(d-1)n},q^r/b;q^d)_k}{(bq^{d-r},q^d;q^d)_k}q^{dk}\notag\\
	&\quad=\frac{(b^2q^{d-2r};q^d)_{n-1-\frac{n-r}{d}}}{(q^d;q^{d})_{n-1-\frac{n-r}{d}}}(q^r/b)^{n-1-\frac{n-r}{d}},
\end{align*}
where the last identity follows from $q\mapsto q^d$, $n\mapsto n-1-\frac{n-r}{d}$, $a\mapsto q^r/b$, $c\mapsto bq^{d-r}$ in \eqref{identiti2}. Moreover, since $(q^{d-r-(d-1)n};q^{d})_{i+j}=0$ when $n-1-\frac{n-r}{d}<i+j\leq n-1$, we can conclude that the $q$-congruence \eqref{qcona1} modulo $a-q^{(d-1)n}$ holds.

	Secondly, for $b=q^n$, 	performing the replacements $q\mapsto q^d$, and taking $q^a\mapsto q^{d-r}/a$, $q^b\mapsto q^{r-n}$, $q^{b'}\mapsto q^{r-n}$, $q^{c}\mapsto q^d$, $x\mapsto q^d$ in \eqref{identiti1}, we obtain 
\begin{align*}
	&\sum_{i=0}^{\frac{n-r}{d}}\sum_{j=0}^{\frac{n-r}{d}}\frac{(q^{d-r}/a;q^{d})_{i+j}(q^{r-n};q^{d})_{i}(q^{r-n};q^{d})_{j}}{(q^d;q^d)_{i+j}(q^d;q^d)_i(q^d;q^d)_j}q^{di}(aq^{n})^j\notag\\
	&\quad=\frac{(q^{d-r+n},aq^{r};q^{d})_\infty}{(q^d,aq^{n};q^{d})_\infty}\sum_{k=0}^{\infty}\frac{(q^{d-r}/a,q^{r-n};q^d)_k}{(q^{d-r+n},q^d;q^d)_k}q^{dk}\notag\\
	&\quad=\frac{(aq^{r};q^d)_{\frac{n-r}{d}}(aq^n;q^d)_{\frac{n-r}{d}}}{(q^d;q^{d})_\frac{n-r}{d}(q^{d-r+n};q^d)_{\frac{n-r}{d}}}(q^{d-r}/a)^{\frac{n-r}{d}}.
\end{align*}	
The last identity follows from $q\mapsto q^d$, $n\mapsto \frac{n-r}{d}$, $a\mapsto q^{d-r}/a$, $c\mapsto q^{d-r+n}$ in \eqref{identiti2}. In other words, the $q$-congruence \eqref{qcona1} is valid modulo	$b-q^n$ by applying the fact that $(q^{r-n};q^{d})_{i}=0$ or $(q^{r-n};q^{d})_{j}=0$ for $ \frac{2(n-r)}{d}\leq i+j \leq n-1$.

Finally, since $a-q^{(d-1)n}$ are coprime with $b-q^n$, by applying the Chinese remainder theorem and the following formula:
\begin{align*}
	\frac{b^{d-1}-q^{(d-1)n}}{b^{d-1}-a}\equiv 1\pmod{a-q^{(d-1)n}},\quad\quad\quad\frac{a-q^{(d-1)n}}{a-b^{d-1}}\equiv 1\pmod{b-q^n}, 
\end{align*}
we obtain the $q$-congruence \eqref{qcona1}.

\end{proof}
Finally, the following lemma shall also be necessary. For a short proof of it, we refer the reader to \cite[Lemma 2.1]{GuoZhu2020}.
\begin{lemma}\label{Th1proofLemma1}
	Let $n, d$ be  positive integers with $\gcd{(n,d)}=1$. Then,
	\begin{align*}
		(q^d;q^d)_{n-1}&\equiv n\pmod{\Phi_{n}(q)}.
	\end{align*}
\end{lemma}

\begin{proof}[Proof of \eqref{thm-q-01}]
	  Letting  $b=1$ in  the $(d,r)=(2,1)$ case of \eqref{qlemma11cong11111}, we obtain the following $q$-congruence: modulo $\Phi_{n}(q)(a-q^n)$,
\begin{align}\label{thm-acong1}
	\Phi^{(1)}[q/a;q,q;q^2;q^2;q^2,a]_{\frac{n-1}{2}}&\equiv \frac{(aq;q^2)_{\frac{n-1}{2}}(a;q^2)_{\frac{n-1}{2}}}{(q^2;q^2)_{\frac{n-1}{2}}(q;q^2)_{\frac{n-1}{2}}}(q/a)^{\frac{n-1}{2}}\times \frac{a-q^n}{a-1}\notag\\
	&=-(a-q^n)\frac{(aq;q^2)_{\frac{n-1}{2}}(aq^2;q^2)_{\frac{n-3}{2}}}{(q^2;q^2)_{\frac{n-1}{2}}(q;q^2)_{\frac{n-1}{2}}}(q/a)^{\frac{n-1}{2}}.
\end{align}

	Furthermore, taking $a=1$ in \eqref{thm-acong1} and applying the following $q$-congruence: 
\begin{align*}
-q^{\frac{n-1}{2}}(1-q^n)\frac{(q^2;q^2)_{\frac{n-3}{2}}}{(q^2;q^2)_{\frac{n-1}{2}}}=-q^{\frac{n-1}{2}}\frac{1-q^n}{1-q^{n-1}}\equiv q^{\frac{n+1}{2}}[n]\pmod{\Phi_{n}(q)^2},
\end{align*}
we conclude that the $q$-congruence \eqref{thm-q-01} is true modulo  $\Phi_{n}(q)^2$. The final step is to prove that the $q$-congruence \eqref{thm-q-01} is valid modulo $[n]$.

Based on the preceding proof, we have obtained the following congruence:
\begin{align}\label{thm-acong2}
	\sum_{i=0}^{\frac{n-1}{2}}\sum_{j=0}^{\frac{n-1}{2}}\frac{(q;q^{2})_{i+j}(q;q^{2})_{i}(q;q^{2})_{j}}{(q^2;q^{2})_{i+j}(q^2;q^2)_i(q^2;q^2)_j}q^{2i}\equiv0 \pmod{\Phi_n(q)}.
\end{align}

	Let $\zeta\neq 1$ be a primitive root of unity of degree $v$ with $v|n$ and $v>1.$ Let $c_q(i,j)$ be the $(i,j)$-th term on the left-hand side of \eqref{thm-acong2}, i.e.,
$$
c_q(i,j)=\frac{(q;q^{2})_{i+j}(q;q^{2})_{i}(q;q^{2})_{j}}{(q^2;q^{2})_{i+j}(q^2;q^2)_i(q^2;q^2)_j}q^{2i}.
$$
Moreover, since $(\zeta;\zeta^{2})_{k}=0$ when $k>\frac{v-1}{2}$, the $q$-congruences \eqref{thm-acong2} with $n=v$ shows that
$$\sum_{i=0}^{v-1}\sum_{j=0}^{v-1}c_\zeta(i,j)=\sum_{i=0}^{\frac{v-1}{2}}\sum_{j=0}^{v-1}c_\zeta(i,j)=\sum_{i=0}^{v-1}\sum_{j=0}^{\frac{v-1}{2}}c_\zeta(i,j)=\sum_{i=0}^{\frac{v-1}{2}}\sum_{j=0}^{\frac{v-1}{2}}c_\zeta(i,j)=0,$$
and we can obviously obtain
$$c_\zeta(kv+i,lv+j)=c_\zeta(kv,lv)c_\zeta(i,j).$$
Thus, 
\begin{align*}
\sum_{i=0}^{\frac{n-1}{2}}\sum_{j=0}^{\frac{n-1}{2}}c_\zeta(i,j)&=\sum_{k=0}^{\frac{n-3v}{2v}}\sum_{l=0}^{\frac{n-3v}{2v}}\sum_{i=0}^{v-1}\sum_{j=0}^{v-1}c_\zeta(kv+i,lv+j)+\sum_{k=0}^{\frac{n-3v}{2v}}\sum_{i=0}^{v-1}\sum_{j=0}^{\frac{v-1}{2}}c_\zeta(kv+i,\frac{n-v}{2}+j)\notag\\
&+\sum_{l=0}^{\frac{n-3v}{2v}}\sum_{i=0}^{\frac{v-1}{2}}\sum_{j=0}^{v-1}c_\zeta(\frac{n-v}{2}+i,lv+j)+\sum_{i=0}^{\frac{v-1}{2}}\sum_{j=0}^{\frac{v-1}{2}}c_\zeta(\frac{n-v}{2}+i,\frac{n-v}{2}+j)\\
&=\sum_{k=0}^{\frac{n-3v}{2v}}\sum_{l=0}^{\frac{n-3v}{2v}}c_\zeta(kv,lv)\sum_{i=0}^{v-1}\sum_{j=0}^{v-1}c_\zeta(i,j)+\sum_{k=0}^{\frac{n-3v}{2v}}c_\zeta(kv,\frac{n-v}{2})\sum_{i=0}^{v-1}\sum_{j=0}^{\frac{v-1}{2}}c_\zeta(i,j)\notag\\
&+\sum_{l=0}^{\frac{n-3v}{2v}}c_\zeta(\frac{n-v}{2},lv)\sum_{i=0}^{\frac{v-1}{2}}\sum_{j=0}^{v-1}c_\zeta(i,j)+c_\zeta(\frac{n-v}{2},\frac{n-v}{2})\sum_{i=0}^{\frac{v-1}{2}}\sum_{j=0}^{\frac{v-1}{2}}c_\zeta(i,j)\\
&=0.
\end{align*}
This shows that the left-hand side of the $q$-congruence \eqref{thm-q-01} is divisible by $\Phi_v(q)$, where $v>1$ and $v\mid n$. The proof then follows the fact that $\prod_{v\mid n,v>1}\Phi_{v}(q)=[n].$ Then the derivation of \eqref{thm-q-01} follows from the fact that $\mathrm{lcm}(\Phi_n(q)^2,[n])=[n]\Phi_n(q) $.
\end{proof}
\begin{proof}[Proof of \eqref{thm-q-02}]  	Since $\gcd{(n,d)}=1$ and $1\leq \frac{n-r}{d} \leq \frac{n-1}{2}$,  we conclude that  the denominator on the left side of \eqref{thm-q-02} is coprime to $\Phi_n(q)$. Similarly as before, when  $d>2$, choosing $b=1$ in \eqref{qlemma11cong11111}, we have, modulo $\Phi_n(q)(a-q^n)$,
\begin{align}\label{thm-acong3}
	&\Phi^{(1)}[q^r/a;q^r,q^r;q^d;q^d;q^d,aq^{d-2r}]_{\frac{n-r}{d}}\notag\\
	&\quad\equiv \frac{(q^{d-2r};q^d)_{\frac{n-r}{d}}}{(q^d;q^d)_{\frac{n-r}{d}}}q^{\frac{r(n-r)}{d}}\times \frac{1-q^n}{1-a}\notag\\
	&\quad+\frac{(aq^{d-r};q^d)_{\frac{n-r}{d}}(aq^{d-2r};q^d)_{\frac{n-r}{d}}}{(q^d;q^d)_{\frac{n-r}{d}}(q^{d-r};q^d)_{\frac{n-r}{d}}}(q^r/a)^{\frac{n-r}{d}}\times \frac{a-q^n}{a-1}\notag\\
		&\quad\equiv \frac{(q^{d-2r};q^d)_{\frac{n-r}{d}}}{(q^d;q^d)_{\frac{n-r}{d}}}q^{\frac{r(n-r)}{d}}\times \frac{1-a+a-q^n}{1-a}\notag\\
	&\quad+\frac{(aq^{d-r};q^d)_{\frac{n-r}{d}}(aq^{d-2r};q^d)_{\frac{n-r}{d}}}{(q^d;q^d)_{\frac{n-r}{d}}(q^{d-r};q^d)_{\frac{n-r}{d}}}(q^r/a)^{\frac{n-r}{d}}\times \frac{a-q^n}{a-1}\notag\\
&\quad\equiv \frac{(q^{d-2r};q^d)_{\frac{n-r}{d}}}{(q^d;q^d)_{\frac{n-r}{d}}}q^{\frac{r(n-r)}{d}}+	\frac{(a-q^n)q^{\frac{r(n-r)}{d}}}{(q^d;q^d)_{\frac{n-r}{d}}(q^{d-r};q^d)_{\frac{n-r}{d}}}\notag\\
&\quad \times\frac{a^{\frac{n-r}{d}}(q^{d-r};q^d)_{\frac{n-r}{d}}(q^{d-2r};q^d)_{\frac{n-r}{d}}-(aq^{d-r};q^d)_{\frac{n-r}{d}}(aq^{d-2r};q^d)_{\frac{n-r}{d}}}{a^{\frac{n-r}{d}}(1-a)}.
\end{align}
Finally, by applying the L'H\^{o}pital rule and straightforward computation, we obtain
\begin{align*} 
		&\lim_{a\to 1}\frac{(a-q^n)q^{\frac{r(n-r)}{d}}}{(q^d;q^d)_{\frac{n-r}{d}}(q^{d-r};q^d)_{\frac{n-r}{d}}}\frac{a^{\frac{n-r}{d}}(q^{d-r};q^d)_{\frac{n-r}{d}}(q^{d-2r};q^d)_{\frac{n-r}{d}}-(aq^{d-r};q^d)_{\frac{n-r}{d}}(aq^{d-2r};q^d)_{\frac{n-r}{d}}}{a^{\frac{n-r}{d}}(1-a)}\\
		&\quad=\frac{q^{\frac{r(n-r)}{d}}(q^{d-2r};q^d)_{\frac{n-r}{d}}}{(q^d;q^d)_{\frac{n-r}{d}}}\bigg\{-\frac{n-r}{d}(1-q^n)- [n]\sum_{j=1}^{\frac{n-r}{d}}\bigg(\frac{q^{dj-r}}{[dj-r]}+\frac{q^{dj-2r}}{[dj-2r]}\bigg)\bigg\}.
\end{align*}
Hence, by letting $a\to 1$ in \eqref{thm-acong3}, we have proven \eqref{thm-q-02} through the above limit.

\end{proof}
\begin{proof}[Proof of \eqref{thm-q-03}] 
	
Following the line of proof for \eqref{thm-q-02} and letting $b=1$ in \eqref{qcona1}, we get, modulo $\Phi_{n}(q)(a-q^{(d-1)n})$,
\begin{align}\label{thm-q-03-qcona1}
	&\sum_{i=0}^{n-1}\sum_{j=0}^{n-1-i}\frac{(q^{d-r}/a;q^{d})_{i+j}(q^{r};q^{d})_{i}(q^{r};q^{d})_{j}}{(q^d;q^d)_{i+j}(q^d;q^d)_i(q^d;q^d)_j}q^{di}a^j\notag\\
	&\quad\equiv q^{r(n-1-\frac{n-r}{d})}\frac{(q^{d-2r};q^d)_{n-1-\frac{n-r}{d}}}{(q^{d};q^d)_{n-1-\frac{n-r}{d}}}\frac{1-q^{(d-1)n}}{1-a}+(q^{d-r}/a)^{\frac{n-r}{d}}\frac{(aq^{r},a;q^d)_{\frac{n-r}{d}}}{(q^{d},q^{d-r};q^d)_{\frac{n-r}{d}}}\frac{a-q^{(d-1)n}}{a-1}\notag\\
		&\quad= q^{r(n-1-\frac{n-r}{d})}\frac{(q^{d-2r};q^d)_{n-1-\frac{n-r}{d}}}{(q^{d};q^d)_{n-1-\frac{n-r}{d}}}(1+\frac{a-q^{(d-1)n}}{1-a})+(q^{d-r}/a)^{\frac{n-r}{d}}\frac{(aq^{r},a;q^d)_{\frac{n-r}{d}}}{(q^{d},q^{d-r};q^d)_{\frac{n-r}{d}}}\frac{a-q^{(d-1)n}}{a-1}\notag\\
			&\quad\equiv q^{r(n-1-\frac{n-r}{d})}\frac{(q^{d-2r};q^d)_{n-1-\frac{n-r}{d}}}{(q^{d};q^d)_{n-1-\frac{n-r}{d}}}- (q^{d-r}/a)^{\frac{n-r}{d}}(a-q^{(d-1)n})\frac{(aq^{r};q^d)_{\frac{n-r}{d}}(aq^d;q^d)_{\frac{n-r-d}{d}}}{(q^{d},q^{d-r};q^d)_{\frac{n-r}{d}}},
\end{align}
where the final congruence holds because $(q^{d-2r};q^d)_{n-1-\frac{n-r}{d}}$ contains the factor $1-q^{(d-2)n}$.

Next, by setting $a=1$ in \eqref{thm-q-03-qcona1} and making use of Lemma \ref{Th1proofLemma1} and the congruences below:
\begin{align*}
q^{r(n-1-\frac{n-r}{d})}\frac{(q^{d-2r};q^d)_{n-1-\frac{n-r}{d}}}{(q^{d};q^d)_{n-1-\frac{n-r}{d}}}&=q^{r(n-1-\frac{n-r}{d})}\frac{(q^{d-2r};q^d)_{n-1}(q^{dn-n+r};q^d)_{\frac{n-r}{d}}}{(q^d;q^d)_{n-1}(q^{-r+(d-1)n};q^d)_{\frac{n-r}{d}}}\notag\\
&\equiv q^{-r-r(n-r)/d}\frac{(q^{d-2r};q^d)_{n-1}}{(q^d;q^d)_{n-1}}\frac{(q^{r};q^d)_{\frac{n-r}{d}}}{(q^{-r};q^d)_{\frac{n-r}{d}}}\notag\\
&\equiv -q^{-r(n-r)/d}\frac{(q^{d-2r};q^d)_{n-1}}{n}\frac{(q^{r+d};q^d)_{\frac{n-r-d}{d}}}{(q^{-r+d};q^d)_{\frac{n-r-d}{d}}}\pmod{\Phi_{n}(q)^2},
\end{align*}
and $1-q^{(d-1)n}\equiv (d-1)(1-q^n)\pmod{\Phi_{n}(q)^2}$,
we have demonstrated that $q$-congruence \eqref{thm-q-03} is valid.

\end{proof}

       \section{Proof of Theorem \ref{thm-aaa}}
	\setcounter{lemma}{0}
\setcounter{theorem}{0}
\setcounter{equation}{0}
\setcounter{conjecture}{0}
\setcounter{remark}{0}
\setcounter{corollary}{0}       
       
      In this section, we prove Theorem \ref{thm-aaa} via the $q$-Vandermonde summation formula \eqref{identiti2} and the Andrews' terminating $q$-analogue of Watson's ${}_{3}F_2$ sum \eqref{identiti4}. Additionally, we propose a related conjecture. Prior to the proof, we establish the following auxiliary result.
 \begin{lemma}\label{thm-aaa-Lemma1}
 	Let $n$ be an odd positive integer and let $j$ be an integer with $0\leq j\leq \frac{n-1}{2}$.   Then, 
 	\begin{align}\label{thm-aaa-Lemma1-qcon1}
 		\frac{(q;q^2)_{\frac{n-1}{2}-j}}{(q^2;q^2)_{\frac{n-1}{2}-j}}q^{\frac{n-1}{2}-j}=\frac{(q^{2-n};q^2)_{\frac{n-1}{2}}(q^{1-n};q^2)_j}{(q^{1-n};q^2)_{\frac{n-1}{2}}(q^{2-n};q^2)_j}.
 	\end{align}
 \end{lemma}     
 \begin{proof}  [Sketch of proof.]   
A straightforward derivation yields:
\begin{align}\label{thm-aaa-Lemma1-qcon2}
(q;q^2)_{\frac{n-1}{2}-j}&=(-1)^{\frac{n-1}{2}-j}(q^{2j+2-n};q^2)_{\frac{n-1}{2}-j}q^{(n-2j-1)^2/4}\notag\\
&=(-1)^{\frac{n-1}{2}-j}\frac{(q^{2-n};q^2)_{\frac{n-1}{2}}}{(q^{2-n};q^2)_{j}}q^{(n-2j-1)^2/4},
\end{align}
and 
\begin{align}\label{thm-aaa-Lemma1-qcon3}
(q^2;q^2)_{\frac{n-1}{2}-j}&=(-1)^{\frac{n-1}{2}-j}(q^{2j+1-n};q^2)_{\frac{n-1}{2}-j}q^{((n-2j)^2-1)/4}\notag\\
&=(-1)^{\frac{n-1}{2}-j}\frac{(q^{1-n};q^2)_{\frac{n-1}{2}}}{(q^{1-n};q^2)_{j}}q^{((n-2j)^2-1)/4}.
\end{align}
By combining the two aforementioned relations \eqref{thm-aaa-Lemma1-qcon2} and \eqref{thm-aaa-Lemma1-qcon3}, \eqref{thm-aaa-Lemma1-qcon1} follows directly.

\end{proof}      
      
       \begin{proof}[Proof of Theorem \ref{thm-aaa}]
      By applying the formula $q^n\equiv 1\pmod{\Phi_{n}(q)}$ and Lemma \ref{thm-aaa-Lemma1}, we derive, modulo $\Phi_{n}(q)$,
  	\begin{align*}
  	&\sum_{i=0}^{\frac{n-1}{2}}\sum_{j=0}^{\frac{n-1}{2}}\frac{(q;q^{2})_{i+j}(q;q^{2})_{i}(q,-q;q^{2})_{j}}{(q^2;q^2)_i^2(q^2;q^2)_j^2(-q^2;q^2)_j}q^{2i+2j}\notag\\
  	&\quad\equiv \sum_{i=0}^{\frac{n-1}{2}}\sum_{j=0}^{\frac{n-1}{2}}\frac{(q^{1-n};q^{2})_{i+j}(q;q^{2})_{i}(q,-q;q^{2})_{j}}{(q^2;q^2)_i^2(q^2;q^2)_j^2(-q^{2-n};q^2)_j}q^{2i+2j}\notag\\
  	&\quad= \sum_{j=0}^{\frac{n-1}{2}}\frac{(q^{1-n},q,-q;q^{2})_{j}q^{2j}}{(q^2;q^2)_j^2(-q^{2-n};q^2)_j}\sum_{i=0}^{\frac{n-1}{2}}\frac{(q^{1-n+2j};q^{2})_{i}(q;q^{2})_{i}}{(q^2;q^2)_i^2}q^{2i}\notag\\
  	&\quad= \sum_{j=0}^{\frac{n-1}{2}}\frac{(q^{1-n},q,-q;q^{2})_{j}q^{2j}}{(q^2;q^2)_j^2(-q^{2-n};q^2)_j}\frac{(q;q^{2})_{\frac{n-1}{2}-j}}{(q^2;q^2)_{\frac{n-1}{2}-j}}q^{\frac{n-1}{2}-j}\notag\\	&\quad\equiv  \frac{(q^{2};q^2)_{\frac{n-1}{2}}}{(q;q^2)_{\frac{n-1}{2}}}\sum_{j=0}^{\frac{n-1}{2}}\frac{(q^{1-n},q,-q,q^{1-n};q^{2})_{j}q^{2j}}{(q^2,q^2,-q^{2-n},q^{2-n};q^2)_j}\notag\\	
  		&\quad=\begin{cases}
  		0 &\text{if }n\equiv 3\pmod{4},\\[5pt]
  			\dfrac{(q^{2};q^2)_{\frac{n-1}{2}}}{(q;q^2)_{\frac{n-1}{2}}}\dfrac{q^{\frac{n-1}{2}}(q^2,q^{2-2n};q^4)_{\frac{n-1}{4}}}{(q^{4-2n},q^4;q^4)_{\frac{n-1}{4}}}&\text{if }n\equiv 1\pmod{4},
  		\end{cases}\notag\\	
  		&\quad\equiv \begin{cases}
  			0 &\text{if }n\equiv 3\pmod{4},\\[5pt]
  			\dfrac{q^{3(n-1)/4}(q^{2};q^4)_{\frac{n-1}{4}}^2}{(q^4;q^4)_{\frac{n-1}{4}}^2}&\text{if }n\equiv 1\pmod{4},
  		\end{cases}
  		 \end{align*}     	
where in the third equality utilizes formula \eqref{identiti2} with $n \mapsto \frac{n-1}{2}-j, q\mapsto q^2,a\mapsto q, c\mapsto q^2$ and in the second-to-last equality we have used  	\eqref{identiti4} with $n \mapsto \frac{n-1}{2}, q\mapsto q^2,a\mapsto q^{-n}, c\mapsto q $.	 
This leads us directly to \eqref{qconaaa1}.

       \end{proof}
Based on comprehensive numerical evidence, we propose the following conjectural statement, which provides a $q$-analogue of the first case of \eqref{Appellseries07}. 
\begin{conjecture}
	Let $n$ be a positive odd integer with $n\equiv 3\pmod{4}$. Then,
\begin{align*}
	\sum_{i=0}^{\frac{n-1}{2}}\sum_{j=0}^{\frac{n-1}{2}}\frac{(q;q^{2})_{i+j}(q;q^{2})_{i}(q;q^{2})_{j}(-q;q^{2})_{j}}{(q^2;q^2)_i^2(q^2;q^2)_j^2(-q^2;q^2)_j}q^{2i+2j}\equiv0\pmod{\Phi_{n}(q)^2}.
\end{align*}	                           
\end{conjecture}

       \section{Proof of Theorem \ref{thm-heliu02}}
       	\setcounter{lemma}{0}
       \setcounter{theorem}{0}
       \setcounter{equation}{0}
       \setcounter{conjecture}{0}
       \setcounter{remark}{0}
       \setcounter{corollary}{0}
       Following the methodology of Theorem \ref{thm-aaa}'s proof, this section employs the  $q$-Chu-Vandermonde identity \eqref{identiti3} in conjunction with the Andrews' terminating $q$-analogue of Watson's ${}_{3}F_2$ sum \eqref{identiti4} to establish the result.

       \begin{proof}
       	Due to the following two formulas, 
     $$(q^{1-n},q^{1+n};q^2)_k\equiv(q,q;q^2)_k\pmod{\Phi_{n}(q)^2},$$  and
     $$\sum_{j=0}^kq^{2j^2}\qbinom{k}{j}{q^2}^2=\sum_{j=0}^kq^{2j^2}\qbinom{k}{j}{q^2}\qbinom{k}{k-j}{q^2}=\qbinom{2k}{k}{q^2},$$
 we obtain, modulo $\Phi_{n}(q)^2$, 
 
\begin{align*}
	&\sum_{i=0}^{\frac{n-1}{2}}\sum_{j=0}^{\frac{n-1}{2}}\frac{(q;q^{2})_{i+j}^2q^{2i+2j+2j^2}}{(q^2;q^2)_i^2(q^2;q^2)_j^2(-q^2;q^2)_{i+j}^2}\notag\\
	&\quad\equiv \sum_{i=0}^{\frac{n-1}{2}}\sum_{j=0}^{\frac{n-1}{2}}\frac{(q;q^{2})_{i+j}^2}{(q^2;q^2)_{i+j}^2(-q^2;q^2)_{i+j}^2}\frac{(q^2;q^2)_{i+j}^2q^{2i+2j+2j^2}}{(q^2;q^2)_{i}^2(q^2;q^2)_j^2}\notag\\
	&\quad= \sum_{k=0}^{n-1}\sum_{j=0}^k\frac{(q;q^{2})_{k}^2q^{2k}}{(q^2;q^2)_{k}^2(-q^2;q^2)_{k}^2}\frac{(q^2;q^2)_{k}^2q^{2j^2}}{(q^2;q^2)_{k-j}^2(q^2;q^2)_j^2}\notag\\
	&\quad= \sum_{k=0}^{n-1}\frac{(q;q^{2})_{k}^2q^{2k}}{(q^2;q^2)_{k}^2(-q^2;q^2)_{k}^2}\sum_{j=0}^{k}\frac{(q^2;q^2)_{k}^2q^{2j^2}}{(q^2;q^2)_{k-j}^2(q^2;q^2)_j^2}\notag\\	&\quad\equiv  \sum_{k=0}^{(n-1)/2}\frac{(q^{1-n},q^{1+n};q^{2})_{k}q^{2k}}{(q^2;q^2)_{k}^2(-q^2;q^2)_{k}^2}\sum_{j=0}^{k}q^{2j^2}\qbinom{k}{j}{q^2}\qbinom{k}{k-j}{q^2}\notag\\
	&\quad=  \sum_{k=0}^{(n-1)/2}\frac{(q^{1-n},q^{1+n};q^{2})_{k}q^{2k}}{(q^2;q^2)_{k}^2(-q^2;q^2)_{k}^2}\qbinom{2k}{k}{q^2}\notag\\
	&\quad=  \sum_{k=0}^{(n-1)/2}\frac{(q^{1-n},q^{1+n};q^{2})_{k}(q^2;q^2)_{2k}q^{2k}}{(q^2;q^2)_{k}^4(-q^2;q^2)_{k}^2}\notag\\	
	&\quad=  \sum_{k=0}^{(n-1)/2}\frac{(q^{1-n},q^{1+n};q^{2})_{k}(q;q^2)_k(-q;q^2)_kq^{2k}}{(q^2;q^2)_{k}^3(-q^2;q^2)_{k}}\notag\\	
	&\quad\equiv \begin{cases}
		0 &\text{if }n\equiv 3\pmod{4},\\[5pt]
		\dfrac{q^{(n-1)/2}(q^{2};q^4)_{\frac{n-1}{4}}^2}{(q^4;q^4)_{\frac{n-1}{4}}^2}&\text{if }n\equiv 1\pmod{4}.
	\end{cases}
\end{align*}      
 In deriving the last step, we employ \eqref{identiti4} with $q\mapsto q^2, n \mapsto \frac{n-1}{2}, a\mapsto 1, c\mapsto q $. This concludes the proof of Theorem \ref{thm-heliu02}.

 \end{proof}

 \section{Proof of Theorem \ref{thm-ApaZeil-qcon-12}}  
 	\setcounter{lemma}{0}
 \setcounter{theorem}{0}
 \setcounter{equation}{0}
 \setcounter{conjecture}{0}
 \setcounter{remark}{0}
 \setcounter{corollary}{0}
 In this chapter, via the `creative microscoping' method, we will first present the following 
 $q$-congruence: for arbitrary positive integers  $n,s,r$ with the constraint $r\leq s$,
 \begin{align} \label{thm-ApaZeil-lemma1-con2}
	&\sum_{i=0}^{rn-1}\sum_{j=0}^{sn-1}\frac{(q/a;q)_{i+j}(aq;q)_{i+j}}{(q/a;q)_i(aq;q)_i(q;q)_j^2}q^{i+j}\equiv\bigg(\frac{rn}{3}\bigg)q^{2(r^2n^2-1)/3} \pmod{(a-q^{rn})(1-aq^{rn})}. 
\end{align}   
 \begin{proof}[Sketch of proof.]
For $a=q^{rn}$ or $a=q^{-rn}$,  the left side of \eqref{thm-ApaZeil-lemma1-con2} is equal to:
\begin{align*}
	&\sum_{i=0}^{rn-1}\sum_{j=0}^{sn-1}\frac{(q^{1-rn};q)_{i+j}(q^{1+rn};q)_{i+j}}{(q^{1-rn};q)_i(q^{1+rn};q)_i(q;q)_j^2}q^{i+j}\notag\\
	&=\sum_{i=0}^{rn-1}q^i\sum_{j=0}^{rn-1}\frac{(q^{1-rn+i};q)_{j}(q^{1+rn+i};q)_{j}q^j}{(q;q)_j^2}  \notag\\
	&= \sum_{i=0}^{rn-1}\frac{(q^{-rn-i};q)_{rn-1-i}q^{r^2n^2-(1+i)^2+i}}{(q;q)_{rn-1-i}} \notag\\
	&= \sum_{i=0}^{rn-1}\frac{(q^{1-2rn+i};q)_{i}q^{rn-1+(2rn-1-i)i}}{(q;q)_{i}}=\bigg(\frac{rn}{3}\bigg)q^{2(r^2n^2-1)/3}.
\end{align*}  
In the second step, we employed the $q$-Chu-Vandermonde summation formula \eqref{identiti2}, and in the third step, we replaced $i$ with $rn-1-i$. It is worth noting that the final step utilizes formula \cite[(6.1)]{Hewang2026}, which was proved by He and Wang \cite{Hewang2026} using the $q$-Zeilberger algorithm. Thus, the validity of the equation \eqref{thm-ApaZeil-lemma1-con2} is established. 
\end{proof} 
It is evident that the case $(a,r,s)=(1,1,1)$ of \eqref{thm-ApaZeil-lemma1-con2} reduces to \eqref{Hewang01}, thereby providing a $q$-analogue for \eqref{ApaZeilberger-conje-1}.      
In fact, our attempts to prove Theorem \ref{thm-ApaZeil-qcon-12} via the `creative microscoping' method proved unsuccessful, chiefly due to our inability to demonstrate that the numerator of \eqref{thm-ApaZeil-qcon-2} contains sufficiently high powers of $\Phi_n(q)$. Consequently, we altered our approach and turned to differential operators and root-of-unity properties in order to derive the required identities. Prior to the proof of Theorem \ref{thm-ApaZeil-qcon-12}, we shall require the following three lemmas.

Prior to the proof of Theorem \ref{thm-ApaZeil-qcon-12}, we shall require the $q$-analogue of the classical congruence 
\begin{equation}
	\binom{ap}{bp} \equiv \binom{a}{b} \pmod{p^3}\qquad ( p\geq 5).
\end{equation}
The reader is referred to \cite{An1999, Straub} for further discussion.

\begin{lemma}\cite[ Theorem 2.2]{Straub} \label{L1} For positive integers $n$ and nonnegative integers $A, a$,
	
	\begin{equation}\label{L1eq}
		\qbinom{An}{an}{q}\equiv \qbinom{A}{a}{q^{n^2}}-(A-a)a\binom{A}{a}\frac{n^2-1}{24}(q^n-1)^2
		\pmod{\Phi_n(q)^3}.
	\end{equation}
\end{lemma}

\begin{lemma}\label{L2}
	Define
	\[
	\Lambda_t(q):=\sum_{j=1}^t \frac{q^j}{1-q^j}
	\qquad (t\ge0,\ \Lambda_0(q):=0).
	\]
	Then,
	\begin{equation}\label{L2eq}
		\sum_{m=0}^{n-1} q^m \sum_{b=0}^m
		\begin{bmatrix}
			m\\
			b
		\end{bmatrix}_q^{\!2}
		\Bigl(2\Lambda_b(q)-2\Lambda_m(q)-(n-1)\Bigr)
		\equiv 0 \pmod{\Phi_n(q)}.
	\end{equation}
\end{lemma}

\begin{proof}
	We divide the proof into three steps.

	\textbf{Step 1. } We first show that for every integer $m\ge 0$,
	\begin{equation}
		\sum_{b=0}^m q^{b^2}\qbinom{m}{b}{q}^{\!2}
		\bigl(b+\Lambda_b(q)-\Lambda_{m-b}(q)\bigr)=0.
		\label{eq:aux}
	\end{equation}
	
	To prove this, we start from the $q$-Chu--Vandermonde identity \eqref{identiti2-1}.
	Taking $a\mapsto xq^{-m}$, $c\mapsto xq$, $n\mapsto m$ in \eqref{identiti2-1} gives
	\begin{equation}
		\sum_{b=0}^m
		\frac{(q^{-m};q)_b(xq^{-m};q)_b}{(q;q)_b(xq;q)_b}
		\,q^{(2m+1)b}
		=
		\frac{(q^{m+1};q)_m}{(xq;q)_m}.
		\label{eq:4}
	\end{equation}
	
	We use the standard identities
	\begin{equation}
		(q^{-m};q)_b
		=
		(-1)^b q^{-mb+\binom{b}{2}}
		\frac{(q;q)_m}{(q;q)_{m-b}},
		\label{eq:5}
	\end{equation}
	and
	\begin{equation}
		(xq^{-m};q)_b
		=
		(-x)^b q^{-mb+\binom{b}{2}}
		(x^{-1}q^{m-b+1};q)_b.
		\label{eq:6}
	\end{equation}
	Substituting \eqref{eq:5} and \eqref{eq:6} into \eqref{eq:4}, we obtain
	\begin{equation}
		\sum_{b=0}^m
		x^b q^{b^2}
		\qbinom{m}{b}{q}
		\frac{(x^{-1}q^{m-b+1};q)_b}{(xq;q)_b}
		=
		\frac{(q^{m+1};q)_m}{(xq;q)_m}.
		\label{eq:7}
	\end{equation}
	
	Setting $x=1$ in \eqref{eq:7}, we get
	\begin{equation}
		\sum_{b=0}^m q^{b^2}\qbinom{m}{b}{q}^{\!2}
		=
		\frac{(q^{m+1};q)_m}{(q;q)_m}
		=
		\genfrac{[}{]}{0pt}{}{2m}{m}_q.
		\label{eq:8}
	\end{equation}
	This is likewise a particular instance of \eqref{identiti3}.
	
	Next define
	\[
	f_b(x):=
	x^b\frac{(x^{-1}q^{m-b+1};q)_b}{(xq;q)_b}.
	\]
	Then \eqref{eq:7} becomes
	\[
	\sum_{b=0}^m q^{b^2}\qbinom{m}{b}{q} f_b(x)
	=
	\frac{(q^{m+1};q)_m}{(xq;q)_m}.
	\]
	
	Since both sides are rational functions in $x$, we shall differentiation at $x=1$. 
	It is evident that
	\[
	\left.\frac{d}{dx}\log f_b(x)\right|_{x=1}
	=
	b
	+
	\left.\frac{d}{dx}\log (x^{-1}q^{m-b+1};q)_b\right|_{x=1}
	-
	\left.\frac{d}{dx}\log (xq;q)_b\right|_{x=1}.
	\]
	Now
	\[
	(x^{-1}q^{m-b+1};q)_b
	=
	\prod_{r=m-b+1}^{m}(1-q^r/x),
	\]
	so
	\[
	\left.\frac{d}{dx}\log (x^{-1}q^{m-b+1};q)_b\right|_{x=1}
	=
	\sum_{r=m-b+1}^{m}\frac{q^r}{1-q^r}.
	\]
	Similarly,
	\[
	\left.\frac{d}{dx}\log (xq;q)_b\right|_{x=1}
	=
	-\sum_{j=1}^{b}\frac{q^j}{1-q^j}.
	\]
	Hence
	\begin{equation}
		\left.\frac{d}{dx}\log f_b(x)\right|_{x=1}
		=
		b+\Lambda_m(q)-\Lambda_{m-b}(q)+\Lambda_b(q),
		\label{eq:9}
	\end{equation}
	and therefore
	\begin{equation}
		f_b'(1)=f_b(1)\left.\frac{d}{dx}\log f_b(x)\right|_{x=1}
		=
		\qbinom{m}{b}{q}
		\bigl(b+\Lambda_m(q)-\Lambda_{m-b}(q)+\Lambda_b(q)\bigr).
		\label{eq:10}
	\end{equation}
	
	On the other hand,
	\[
	\frac{(q^{m+1};q)_m}{(xq;q)_m}
	=
	(q^{m+1};q)_m\,(xq;q)_m^{-1},
	\]
	so
	\begin{equation}
		\left.\frac{d}{dx}\frac{(q^{m+1};q)_m}{(xq;q)_m}\right|_{x=1}
		=
		\Lambda_m(q)\frac{(q^{m+1};q)_m}{(q;q)_m}
		=
		\Lambda_m(q)\genfrac{[}{]}{0pt}{}{2m}{m}_q.
		\label{eq:11}
	\end{equation}
	
	Differentiating \eqref{eq:7} at $x=1$ and applying \eqref{eq:10}--\eqref{eq:11}, we obtain
	\begin{equation}
		\sum_{b=0}^m q^{b^2}\qbinom{m}{b}{q}^{\!2}
		\bigl(b+\Lambda_m(q)-\Lambda_{m-b}(q)+\Lambda_b(q)\bigr)
		=
		\Lambda_m(q)\genfrac{[}{]}{0pt}{}{2m}{m}_q.
		\label{eq:12}
	\end{equation}
	Using \eqref{eq:8}, this simplifies to
	\[
	\sum_{b=0}^m q^{b^2}\qbinom{m}{b}{q}^{\!2}
	\bigl(b+\Lambda_b(q)-\Lambda_{m-b}(q)\bigr)=0,
	\]
	which is exactly \eqref{eq:aux}.
	
	\medskip
	\noindent
	\textbf{Step 2. } To prove \eqref{L2eq}, we evaluate it at primitive $n$th roots of unity.
	
	Let $\zeta$ be a primitive $n$th root of unity, and define
	\begin{equation}
		S(\zeta):=
		\sum_{m=0}^{n-1}\zeta^m\sum_{b=0}^m
		\genfrac{[}{]}{0pt}{}{m}{b}_{\zeta}^{\!2}
		\bigl(2\Lambda_b(\zeta)-2\Lambda_m(\zeta)-(n-1)\bigr).
		\label{eq:13}
	\end{equation}
	We shall prove that $S(\zeta)=0$.
	
	For $0\le b\le m\le n-1$, set
	\[
	t:=n-1-m+b.
	\]
	Then $0\le b\le t\le n-1$, and conversely $m=n-1-t+b$.
	
 Since
	\[
	\genfrac{[}{]}{0pt}{}{m}{b}_{\zeta}
	=
	\prod_{j=1}^b\frac{1-\zeta^{m-b+j}}{1-\zeta^j},
	\]
	and
	\[
	1-\zeta^{m-b+j}
	=
	-\zeta^{m-b+j}\bigl(1-\zeta^{-(m-b+j)}\bigr),
	\]
	while $-(m-b+j)\equiv n-m+b-j=t-j+1 \pmod n$, we get
	\begin{equation}
		\genfrac{[}{]}{0pt}{}{m}{b}_{\zeta}
		=
		(-1)^b\zeta^{mb-\binom{b}{2}}
		\genfrac{[}{]}{0pt}{}{t}{b}_{\zeta}.
		\label{eq:14}
	\end{equation}
	
	We also use the standard identity
	\begin{equation}
		\genfrac{[}{]}{0pt}{}{t}{b}_{q^{-1}}
		=
		q^{-b(t-b)}
		\genfrac{[}{]}{0pt}{}{t}{b}_{q}.
		\label{eq:15}
	\end{equation}
	Combining \eqref{eq:14} and \eqref{eq:15}, and using $m=n-1-t+b$ and $\zeta^n=1$, one checks that
	\begin{equation}
		\zeta^m
		\genfrac{[}{]}{0pt}{}{m}{b}_{\zeta}^{\!2}
		=
		\zeta^{-t-1}\zeta^{-b^2}
		\genfrac{[}{]}{0pt}{}{t}{b}_{\zeta^{-1}}^{\!2}.
		\label{eq:16}
	\end{equation}
	
	Next, for any integer $s\ge 0$,
	\begin{equation}
		\Lambda_s(\zeta^{-1})
		=
		\sum_{j=1}^s \frac{\zeta^{-j}}{1-\zeta^{-j}}
		=
		\sum_{j=1}^s\left(-1-\frac{\zeta^j}{1-\zeta^j}\right)
		=
		-s-\Lambda_s(\zeta).
		\label{eq:17}
	\end{equation}
	Also, for $1\le j\le n-1$,
	\[
	\frac{\zeta^j}{1-\zeta^j}+\frac{\zeta^{n-j}}{1-\zeta^{n-j}}=-1,
	\]
	hence for $0\le u\le n-1$,
	\begin{equation}
		\Lambda_{n-1-u}(\zeta)
		=
		\Lambda_u(\zeta)-\frac{n-1}{2}+u.
		\label{eq:18}
	\end{equation}
	
	Now note that $t-b=n-1-m$. By \eqref{eq:17} and \eqref{eq:18},
	\begin{align}
		b+\Lambda_b(\zeta^{-1})-\Lambda_{t-b}(\zeta^{-1})
		&=
		b-b-\Lambda_b(\zeta)+(t-b)+\Lambda_{t-b}(\zeta) \notag\\
		&=
		-\Lambda_b(\zeta)+(n-1-m)+\Lambda_{n-1-m}(\zeta) \notag\\
		&=
		-\Lambda_b(\zeta)+(n-1-m)+\Lambda_m(\zeta)-\frac{n-1}{2}+m \notag\\
		&=
		\Lambda_m(\zeta)-\Lambda_b(\zeta)+\frac{n-1}{2}.
		\label{eq:19}
	\end{align}
	Therefore,
	\begin{equation}
		2\Lambda_b(\zeta)-2\Lambda_m(\zeta)-(n-1)
		=
		-2\bigl(b+\Lambda_b(\zeta^{-1})-\Lambda_{t-b}(\zeta^{-1})\bigr).
		\label{eq:20}
	\end{equation}
	
	Substituting \eqref{eq:16} and \eqref{eq:20} into \eqref{eq:13}, and changing variables from $(m,b)$ to $(t,b)$, we obtain
	\begin{equation}
		S(\zeta)
		=
		-2\sum_{t=0}^{n-1}\zeta^{-t-1}
		\sum_{b=0}^{t}
		\zeta^{-b^2}
		\genfrac{[}{]}{0pt}{}{t}{b}_{\zeta^{-1}}^{\!2}
		\bigl(b+\Lambda_b(\zeta^{-1})-\Lambda_{t-b}(\zeta^{-1})\bigr).
		\label{eq:21}
	\end{equation}
	
	But for each fixed $t$, the inner sum is exactly the identity \eqref{eq:aux} with
	\[
	q=\zeta^{-1},\qquad m=t.
	\]
	This substitution is permissible because $1-\zeta^{-j}\neq 0$ for $1\le j\le n-1$. Hence every inner sum in \eqref{eq:21} vanishes, and so
	\[
	S(\zeta)=0.
	\]
	
	\medskip
	\noindent
	\textbf{Step 3.} Consider
	\[
	F(q):=
	\sum_{m=0}^{n-1} q^m \sum_{b=0}^m \qbinom{m}{b}{q}^{\!2}
	\bigl(2\Lambda_b(q)-2\Lambda_m(q)-(n-1)\bigr).
	\]
	The only denominators appearing in $F(q)$ come from factors $1-q^j$ with $1\le j\le n-1$. Since
	\[
	\gcd(1-q^j,\Phi_n(q))=1 \qquad (1\le j\le n-1),
	\]
	all these denominators are invertible modulo $\Phi_n(q)$. Therefore,  $F(q)$ is well-defined in the quotient ring modulo $\Phi_n(q)$.
	
	Now for every primitive $n$th root of unity $\zeta$, we have shown that
	\[
	F(\zeta)=S(\zeta)=0.
	\]
	Therefore,
	\[
	F(q)\equiv 0 \pmod{\Phi_n(q)},
	\]
	that is,
	\[
	\sum_{m=0}^{n-1} q^m \sum_{b=0}^m \qbinom{m}{b}{q}^{\!2}
	\bigl(2\Lambda_b(q)-2\Lambda_m(q)-(n-1)\bigr)
	\equiv 0 \pmod{\Phi_n(q)}.
	\]
	The proof of Lemma \ref{L2} is now complete.
\end{proof}

\begin{lemma}\label{prop:strict-qLucas}
	Let $a,b,c,d$ be nonnegative integers satisfying $0\le b,d<n$. Set $A=a+c$ and $m=b+d<n$. Then,
	\begin{align}\label{eq:block-sum-main}
		\sum_{\substack{0\le b,d<n\\ b+d<n}}
		q^{An+b+d}
		\begin{bmatrix}
			An+b+d\\
			an+b
		\end{bmatrix}_q^{\!2}
		\equiv
		q^{n^2A}
		\begin{bmatrix}
			A\\
			a
		\end{bmatrix}_{q^{n^2}}^{\!2}
		\sum_{\substack{0\le b,d<n\\ b+d<n}}
		q^{b+d}
		\begin{bmatrix}
			b+d\\
			b
		\end{bmatrix}_q^{\!2}
		\pmod{\Phi_n(q)^2}.
	\end{align}
	
\end{lemma}

\begin{proof}
 We start from the exact decomposition
	\begin{align*}
	\begin{bmatrix}
		An+m\\
		an+b
	\end{bmatrix}_q
	=
	\begin{bmatrix}
		An\\
		an
	\end{bmatrix}_q
	\begin{bmatrix}
		m\\
		b
	\end{bmatrix}_q
	\frac{(q^{An+1};q)_m/(q;q)_m}
	{\bigl((q^{an+1};q)_b/(q;q)_b\bigr)\bigl((q^{cn+1};q)_d/(q;q)_d\bigr)}.
\end{align*}
	Since $m,b,d<n$, all factors appearing in the denominator are coprime to  $\Phi_n(q)$. For
	$\Lambda_t(q):=\sum_{j=1}^t \frac{q^j}{1-q^j},$
	we have, for $0\le t<n$,
\begin{align*}
	\frac{(q^{sn+1};q)_t}{(q;q)_t}
	=
	\prod_{j=1}^t \frac{1-q^{sn+j}}{1-q^j}
	\equiv
	1-s(q^n-1)\Lambda_t(q)
	\pmod{\Phi_n(q)^2},
\end{align*}
	since $q^{sn}-1\equiv s(q^n-1)\pmod{\Phi_n(q)^2}$. Hence
	\begin{align*}
	\begin{bmatrix}
		An+m\\
		an+b
	\end{bmatrix}_q
	\equiv
	\begin{bmatrix}
		An\\
		an
	\end{bmatrix}_q
	\begin{bmatrix}
		m\\
		b
	\end{bmatrix}_q
	\Bigl(1+(q^n-1)\bigl(a\Lambda_b(q)+c\Lambda_d(q)-A\Lambda_m(q)\bigr)\Bigr)
	\pmod{\Phi_n(q)^2}.
\end{align*}
	Furthermore, applying Lemma \ref{L1} yields
	\begin{equation}\label{eq:first-order-local}
		\begin{bmatrix}
			An+m\\
			an+b
		\end{bmatrix}_q
		\equiv
		\begin{bmatrix}
			A\\
			a
		\end{bmatrix}_{q^{n^2}}
		\begin{bmatrix}
			m\\
			b
		\end{bmatrix}_q
		\Bigl(1+(q^n-1)\bigl(a\Lambda_b(q)+c\Lambda_d(q)-A\Lambda_m(q)\bigr)\Bigr)
		\pmod{\Phi_n(q)^2}.
	\end{equation}
	
	Squaring \eqref{eq:first-order-local} and multiplying by $q^{An+m}$, noting that
	\begin{align*}
	q^{An}\equiv 1+A(q^n-1)\pmod{\Phi_n(q)^2},
\end{align*}
	we get
		\begin{align*}
	q^{An+m}
	\begin{bmatrix}
		An+m\\
		an+b
	\end{bmatrix}_q^{\!2}
	\equiv
	q^m
	\begin{bmatrix}
		A\\
		a
	\end{bmatrix}_{q^{n^2}}^{\!2}
	\begin{bmatrix}
		m\\
		b
	\end{bmatrix}_q^{\!2}
	\Bigl(1+(q^n-1)\Xi_{b,d}(q)\Bigr)
	\pmod{\Phi_n(q)^2},
\end{align*}
	where
	$
	\Xi_{b,d}(q):=A+2a\Lambda_b(q)+2c\Lambda_d(q)-2A\Lambda_m(q).
	$
	On the other hand,
	$
	q^{n^2A}\equiv 1+nA(q^n-1)\pmod{\Phi_n(q)^2}.
	$
	Therefore \eqref{eq:block-sum-main} reduces to showing that
		\begin{align*}
	\sum_{\substack{0\le b,d<n\\ b+d<n}}
	q^m
	\begin{bmatrix}
		m\\
		b
	\end{bmatrix}_q^{\!2}
	\bigl(\Xi_{b,d}(q)-nA\bigr)
	\equiv 0
	\pmod{\Phi_n(q)}.
\end{align*}
	Since $A=a+c$, this is
\begin{align*}
	\sum_{\substack{0\le b,d<n\\ b+d<n}}
	q^m
	\begin{bmatrix}
		m\\
		b
	\end{bmatrix}_q^{\!2}
	&\Bigl\{a(2\Lambda_b(q)-2\Lambda_m(q)-(n-1))\\
	&+c(2\Lambda_d(q)-2\Lambda_m(q)-(n-1))\Bigr\}
	\equiv 0
	\pmod{\Phi_n(q)}.
\end{align*}
	Now use the symmetry $d=m-b$ and
\begin{align*}
	\begin{bmatrix}
		m\\
		b
	\end{bmatrix}_q=
	\begin{bmatrix}
		m\\
		d
	\end{bmatrix}_q.
\end{align*}
	The left-hand side becomes
\begin{align*}
	A\sum_{m=0}^{n-1} q^m \sum_{b=0}^m
	\begin{bmatrix}
		m\\
		b
	\end{bmatrix}_q^{\!2}
	\Bigl(2\Lambda_b(q)-2\Lambda_m(q)-(n-1)\Bigr),
\end{align*}
	which is congruent to $0\pmod{\Phi_n(q)}$ by Lemma \ref{L2}. This proves
	\eqref{eq:block-sum-main}.
	\end{proof}

We now proceed to prove Theorem \ref{thm-ApaZeil-qcon-12}.
\begin{proof}[Proof of Theorem \ref{thm-ApaZeil-qcon-12}]
Decompose the indices $i,j$ into residue classes modulo $n$:
\[
i=an+b,\qquad j=cn+d,
\]
where
\[
0\le a\le r-1,\quad 0\le c\le s-1,\quad 0\le b,d\le n-1.
\]
Then
\begin{align*}
	\sum_{i=0}^{rn-1}\sum_{j=0}^{sn-1}\frac{(q;q)_{i+j}^2}{(q;q)_i^2(q;q)_j^2}q^{i+j}
	&=\sum_{i=0}^{rn-1}\sum_{j=0}^{sn-1}q^{i+j}\qbinom{i+j}{i}{q}^2\\
	&=\sum_{a=0}^{r-1}\sum_{c=0}^{s-1}
	\sum_{b=0}^{n-1}\sum_{d=0}^{n-1}
	q^{(a+c)n+b+d}\qbinom{(a+c)n+b+d}{an+b}{q}^2.
\end{align*}
By Lemma \ref{prop:strict-qLucas}, for each fixed pair $(a,c)$,
\begin{align}\label{Multiplesum01}
&\sum_{b=0}^{n-1}\sum_{d=0}^{n-1}
q^{(a+c)n+b+d}
\begin{bmatrix}
	(a+c)n+b+d\\
	an+b
\end{bmatrix}_q^{\!2}\notag\\
&\quad\equiv
q^{n^2(a+c)}
\begin{bmatrix}
	a+c\\
	a
\end{bmatrix}_{q^{n^2}}^{\!2}
\sum_{b=0}^{n-1}\sum_{d=0}^{n-1}
q^{b+d}
\begin{bmatrix}
	b+d\\
	b
\end{bmatrix}_q^{\!2}
\pmod{\Phi_n(q)^2}.
\end{align}
Using \eqref{Hewang01} or the case $(a,r,s)=(1,1,1)$ of \eqref{thm-ApaZeil-lemma1-con2}, we obtain
\[
\sum_{b=0}^{n-1}\sum_{d=0}^{n-1}q^{b+d}
\begin{bmatrix}
	b+d\\
	b
\end{bmatrix}_q^{\!2}
\equiv
\left(\frac{n}{3}\right)q^{2(n^2-1)/3}
\pmod{\Phi_n(q)^2}.
\]
Finally, summing the $q$-congruence \eqref{Multiplesum01} over $a$ and $c$ gives
\begin{align*}
	\sum_{i=0}^{rn-1}\sum_{j=0}^{sn-1}q^{i+j}\qbinom{i+j}{i}{q}^2
	&\equiv
	\left(\frac{n}{3}\right)q^{2(n^2-1)/3}
	\sum_{a=0}^{r-1}\sum_{c=0}^{s-1}
	q^{n^2(a+c)}
	\begin{bmatrix}
		a+c\\
		a
	\end{bmatrix}_{q^{n^2}}^{\!2}
	\pmod{\Phi_n(q)^2}.
\end{align*}
 An equivalent form of \eqref{thm-ApaZeil-qcon-2} has now been obtained, and the proof of Theorem \ref{thm-ApaZeil-qcon-12} is finished.

\end{proof}

\end{document}